\documentclass[12pt]{amsart}

\usepackage{environ, xcolor}
\NewEnviron{commentA}{}
\newcommand\Aon{\RenewEnviron{commentA}{\color{purple}\BODY}}

\Aon{} 

\usepackage{amsmath, amssymb}
\usepackage{array}
\usepackage[frame,cmtip,arrow,matrix,line,graph,curve]{xy}
\usepackage{graphpap, color, paralist, pstricks}
\usepackage[mathscr]{eucal}
\usepackage[pdftex]{graphicx}
\usepackage[pdftex,colorlinks,backref=page,citecolor=blue]{hyperref}
\usepackage{cleveref}
\usepackage{tikz-cd, verbatim}
\usepackage{colonequals}

\newtheorem{theorem}{Theorem}[section]
\newtheorem{proposition}[theorem]{Proposition}
\newtheorem{corollary}[theorem]{Corollary}
\newtheorem{lemma}[theorem]{Lemma}

\theoremstyle{definition}
\newtheorem{definition}[theorem]{Definition}
\newtheorem{example}[theorem]{Example}

\newtheorem{question}[theorem]{Question}
\newtheorem{remark}[theorem]{Remark}

\usepackage[margin=1in]{geometry} 
\usepackage{amsmath,amssymb,mathtools, mathrsfs,esint,tikz-cd,verbatim}

\newcommand{\Q}{{\mathbb Q}}

\newcommand{\C}{{\mathbb C}}

\renewcommand{\P}{{\mathbb P}}
\renewcommand{\O}{{\mathcal O}}
\newcommand{\on}[1]{\operatorname{#1}}

\newcommand{\Pic}{{\mathbf{Pic}}}

\newcommand{\Ac}{{\mathcal A}}
\newcommand{\Xc}{{\mathcal X}}
\newcommand{\Lc}{{\mathcal L}}
\newcommand{\irr}{{\on{irr}}}

\DeclareMathOperator{\Bir}{\mathrm{Bir}}

\DeclareMathOperator{\Frac}{\mathrm{Frac}}

\subjclass[2020]{14E05, 14L30}

\title{Group actions and irrationality in surface families}

\author{Nathan Chen}
\address{Department of Mathematics, Harvard University, 1 Oxford Street, Cambridge, MA 02138}
\email{nathanchen@math.harvard.edu}

\author{Louis Esser}
\address{Department of Mathematics, Princeton University, Fine Hall, Washington Road, Princeton, NJ 08544-1000, USA}
\email{esserl@math.princeton.edu}

\begin{document}
\maketitle

\begin{abstract}
Rationality specializes in families of surfaces, even with mild singularities.  In this paper, we study the analogous question for the degree of irrationality. We prove a specialization result when the degree of irrationality on the generic fiber arises from the quotient by a group action.
\end{abstract}

\section{Introduction}

The goal of this paper is to explore some variational properties of measures of irrationality in families. We will focus on the \textit{degree of irrationality}, which for a projective variety $X$ is defined as
\[ \irr(X) \colonequals \min \left\{ \delta > 0 \ \middle| \ \exists \text{ a dominant rational map } X \dashrightarrow \P^{\dim X} \text{ of degree $\delta$ } \right\}. \]
When $X$ is a curve, this coincides with the \textit{gonality} of its normalization.

For a family of integral curves it follows from the classical theory of limit linear series or from \cite[Theorem B]{CS20} that the degree of irrationality cannot jump up along special fibers. However, in higher dimensions it is unclear whether this is always the case, even when we impose smoothness:

\begin{question}\label{ques:inequalitydegirr}
    Let $\Xc \rightarrow B$ be a smooth proper morphism over $\C$ to a smooth connected curve with marked point $0 \in B$. If the very general fiber $\Xc_{b}$ satisfies $\irr(\Xc_{b}) \leq d$, is it true that $\irr(\Xc_{0}) \leq d$?
\end{question}

\noindent When $d = 1$, Question~\ref{ques:inequalitydegirr} has a positive answer by the work of Kontsevich and Tschinkel \cite{KT19}, who used motivic integration techniques to show that rationality specializes in smooth families. For larger values of $d$, very little is known in this direction beyond \cite[Proposition C]{CS20}, where the authors prove that \Cref{ques:inequalitydegirr} also has a positive answer for smooth families of (i) surfaces with $h^{1,0} = 0$, or (ii) strict Calabi-Yau threefolds. That is, the degree of irrationality cannot increase at special points in these types of families. 

Our main result shows the following:

\begin{theorem}
\label{thm:main}
Let $\pi: \Xc \rightarrow B$ be a flat proper morphism over $\C$ to a smooth connected curve $B$. Assume that all of the fibers are integral klt surfaces. If the very general fiber $\Xc_b$ of $\pi$ admits a dominant rational map $\Xc_b \dashrightarrow \mathbb{P}^{2}$ of degree $d$ such that the corresponding extension of function fields is
Galois, then $\mathrm{irr}(\Xc_b) \leq d$ for all $b \in B$.
\end{theorem}

\noindent In other words, if the degree of irrationality of the very general fiber can be birationally computed as a Galois cover, then the degree of irrationality can only go down under specialization. From the point of view of singularities, \Cref{thm:main} is sharp in the sense that if we weaken ``klt" to ``log canonical," the statement is false, even when $d = 1$. For example, consider a pencil of smooth cubic surfaces specializing to a cone over an elliptic curve.  Every smooth cubic surface over $\C$ is rational, while the special fiber is log canonical but not rational (see also \Cref{spec_rationality}).

Since degree $2$ maps are always Galois (birationally), our theorem also implies in particular that the property of having degree of irrationality at most $2$ specializes in families of klt surfaces (\Cref{cor:spec_irr2}). On the other hand, it is plausible that \Cref{thm:main} fails for $d \geq 3$ without the Galois assumption; for instance, there are candidate (but unproven) examples of smooth families of surfaces where $\irr(X)$ may jump from $3$ on the general fiber to $4$ on the special fiber (see \Cref{abel_surf}).

Note that the degree of irrationality can in fact drop at special points in smooth families of any dimension. For example, consider a smooth family of curves specializing to a curve with lower gonality or a generic degree $d \geq 4$ surface in $\P^{3}$ specializing to one containing two disjoint lines. 

The key idea behind \Cref{thm:main} is to specialize the rational $G$-action to the special fiber
and analyze the geometry of the quotient using the $G$-action on the cohomology of the structure sheaf of the original family.
In particular, we show that the $G$-invariant subspace of certain cohomology
groups is constant in klt family of varieties of \textit{any}
dimension with a 
rational $G$-action (\Cref{ginvariants}). In smooth families, this also holds for $G$-invariant global differential forms and pluricanonical forms (\Cref{rem:forms}). 

\subsection*{Conventions and Preliminaries}

Throughout we work over $\C$. $R$ will denote a DVR with field of fractions $K = \text{Frac} \ R$ and residue field $k$ of characteristic $0$.

By \textit{variety}, we mean an integral separated scheme of finite type over a field. A \textit{curve} is a one-dimensional variety, and a \textit{surface} is a two-dimensional variety. A variety $X$ is \textit{ruled} if it is birational to $\P^{1} \times Y$. We will use $(B, 0)$ to denote a pointed smooth connected curve, and $b \in B$ will always mean a closed point.

\subsection*{Acknowledgements} We are grateful to Dori Bejleri, Benjamin Church, Johan de Jong, Christopher Hacon, Andres Fernandez Herrero, Lena Ji, S\'{a}ndor Kov\'{a}cs, John Sheridan,  David Stapleton, and Burt Totaro for valuable discussions. We would also like to thank J\'{a}nos Koll\'{a}r for suggesting an approach to Lemma~\ref{uniruledmod}.

\section{Group quotients and invariant sections}

The purpose of this section is to prove a series of results necessary for the proof
of \Cref{thm:main}.  These culminate in \Cref{ginvariants},
which proves that $G$-invariants of certain cohomology groups
under rational actions have constant dimension in families.
This holds for fibers of any dimension and may be useful in other contexts.

We will frequently use the fact that klt varieties have rational singularities \cite[Theorem 5.22]{KM98}. In order to discuss specialization results for group actions, we will need to recall the following singularity notion \cite[Def. 1.1, Def. 1.5]{CS24}:

\begin{definition}\label{defn:sustained-modifications}
A normal scheme \(X\) has \emph{(uni)ruled modifications} if every exceptional divisor of every normal birational modification \(Y\to X\) is (uni)ruled.
A normal scheme $X_R$ over a DVR $R$ has \textit{sustained (uni)ruled modifications} if there exists a generically finite extension of DVRs $R\subseteq R'$ such that for every generically finite extension of DVRs $R'\subseteq S$, the normalization of $X_{S}$ has (uni)ruled modifications. Here, the ring extension \(R\subseteq R'\) being generically finite means that \(\Frac R'\) is a finite algebraic extension of \(K\). For a family $\pi \colon \Xc \rightarrow B$ to have (sustained) (uni-)ruled modifications at $b \in B$ means that after localizing at $b$ (with local ring $R$), the restriction of the family $\pi$ to $\Xc_{R}$ has the corresponding property.
\end{definition}

We use the following conventions for birational actions.

\begin{definition}
Given a morphism of integral $F$-schemes $f: X \rightarrow T$, where $F$ is a field,
the group $\mathrm{Bir}_{F}(X/T)$ of \textit{birational automorphisms} 
of $X/T$ is the group of birational maps $X \dashrightarrow X$ which 
commute with $f$.  A \textit{birational action} of a group $G$ on $X/T$ 
(or a birational $G$-action) is an abstract injective
group homomorphism $G \rightarrow \mathrm{Bir}_F(X/T)$.
\end{definition}

Throughout the paper, we treat $\Bir_F(X/T)$ as an abstract group; we
also often suppress the subscript $F$, where $F = \C$ is understood.
Finally, we note the following fact about regularization of 
rational actions by finite groups.

\begin{proposition}
\label{prop:regularization}
Let $X \rightarrow T$ be a proper morphism of varieties over a field 
$F$ of characteristic zero. Suppose $G$ is a finite group with a birational action on $X/T$. Then there
is a $G$-equivariant proper birational morphism $Y \rightarrow X$ such that 
$Y$ is smooth and the $G$-action on $Y$ is regular.
\end{proposition}

\noindent This follows from \cite[Section 1.3]{Kraft18} and equivariant resolution of singularities
in characteristic zero.  We will make use of this regularization result 
without further comment.

Our first result shows that the degree of irrationality can only decrease 
under specialization in families of surfaces with rational singularities
such that the special fiber is ruled.

\begin{lemma}
\label{ruled}
Let $\pi: \Xc \rightarrow B$ be a flat proper morphism such that every fiber is an integral surface with rational singularities. 
Suppose that some fiber $\Xc_{0}$
is ruled. If the very general fiber of $\pi$ 
satisfies $\irr(\Xc_b) \leq d$,
then $\irr(\Xc_0) \leq d$.
\end{lemma}

\begin{proof}
Since it does not matter for the statement of the lemma, after a finite
base-change we may assume that $\pi$ has a section. Rational singularities are
Du Bois, so $h^{1}(\O_{\Xc_{b}})$
is independent of $b \in B$ by \cite[Theorem 2.62]{Kollar23}. Thus, under our assumptions,
the Picard scheme $\Pic^{0}(\Xc/B)$ exists and is an abelian scheme 
\cite{Kleiman05}.  The dual abelian scheme (which may be naturally 
constructed as the relative Picard scheme of this family itself) is the
Albanese of the family $\Xc/B$. Using the fact that $\pi$ has a section,
we therefore have an Albanese map of $B$-schemes 
$\mathrm{alb}:\Xc \rightarrow \Pic^{0}(\Xc/B)^{\vee}$.

Note that for a variety $X$ with rational singularities, it is straightforward to check that if $\mu \colon \tilde{X} \rightarrow X$ is a resolution of singularities, then the Albanese varieties for $\tilde{X}$ and $X$ are isomorphic and their Albanese maps commute with $\mu$. Thus, there is no ambiguity when we consider the Albanese of a fiber $\Xc_{b}$ versus that of a resolution.

Next, we may assume that the constant $h^{1}(\O_{\Xc_{b}})$ is positive; otherwise, 
the special fiber is a ruled surface with irregularity zero and 
therefore a rational surface, and the conclusion holds automatically.
Thus, we assume that the Albanese $\Pic^{0}(\Xc/B)^{\vee}$ has 
positive dimension over $B$. Since the special fiber is ruled, 
$\Xc_0$ is birational to $C_0 \times \mathbb{P}^1$ for some smooth 
curve $C_0$ of positive genus. By the previous paragraph, the 
image of the Albanese of $\Xc_0$ is a smooth curve isomorphic to $C_0$.

Since the morphism $\mathcal{C} \coloneqq \mathrm{im}(\mathrm{alb})
\rightarrow B$ is proper
and we know the fiber over $0 \in B$ is isomorphic to $C_0$, by the upper
semicontinuity of fiber dimension and the fact that all fibers have positive
irregularity, the image of the Albanese of any fiber $\Xc_b$ is a curve. 
By generic smoothness, these curves are smooth in a neighborhood of $0 \in B$.  Hence the image 
$\mathcal{C}$ of the Albanese is a family of smooth 
curves over $B$ (after shrinking $B$ if necessary). Gonality can 
only decrease under specialization in smooth families of curves, so
$\irr(\mathcal{C}_b) \geq \irr(C_0)$. We therefore have
\[ \irr(\Xc_b) \geq \irr(\mathcal{C}_b) \geq \irr(C_{0}) = \irr(\Xc_{0}), \]
where the first inequality holds because $\Xc_b$
admits a dominant morphism to the curve $\mathcal{C}_b$ \cite[Theorem 2]{MH82} 
and the last equality is true since 
$\Xc_{0} \simeq_{\text{bir.}} C_{0} \times \P^{1}$ \cite[line 1.1]{CM23}.
\end{proof}

\begin{remark}
\label{spec_rationality}
Note that when $d = 1$, \Cref{ruled} shows that rationality specializes
in families of surfaces with rational singularities.  Indeed, in this case
the special fiber is then ruled by Matsusaka's theorem \cite[Theorem IV.1.6]{Kollar96}
and the irregularity must be zero. 
This specialization property was previously proven for klt surfaces
(and threefolds) \cite[p. 13]{Totaro16}.
Beyond rational singularities, rationality does not specialize in families
of surfaces. For instance, a smooth cubic surface can degenerate
to a cone over an elliptic curve, which is ruled but not rational and
has a log canonical singularity.
\end{remark}

We next consider families of varieties with (rational) $G$-actions
where the special fiber is not ruled. 
We will assume from now on that $G$ is a finite group. 
For the purposes of considering the degree of irrationality,
this is the situation we are most interested in.
We also note in passing that varieties with actions by
positive-dimensional linear algebraic groups are ruled. 
Indeed, if $G$ is a positive-dimensional linear algebraic group
with a faithful rational action on a variety $X$, then $X$ is ruled 
since $G$ contains a copy of a $\mathbb{G}_{\on{a}}$ or a 
$\mathbb{G}_{\on{m}}$ (cf. \cite[IV.1.17.5]{Kollar96}).

Now we are ready to show:

\begin{theorem}
\label{ginvariants}
Let $\pi: \Xc \rightarrow B$ be a flat proper morphism of varieties
such that every fiber is integral with rational singularities. Let $A \subseteq B$ be the
subset of closed points $b \in B$ for which $\Xc_b$ is not uniruled. Suppose $\Xc$ has sustained uniruled modifications over fibers $\Xc_{b}$ with $b \in A$ and there is a birational action $G \subseteq \Bir(\Xc/B)$ by a finite group $G$. Then for every positive 
integer $i$ and every $b \in A$, $G$ acts on $H^i(\Xc_b,\mathcal{O}_{\Xc_b})$ and the 
collection of $G$-representations $\{H^i(\Xc_b,\mathcal{O}_{\Xc_b}): b \in A\}$ is constant 
up to isomorphism.  In particular, $\dim(H^i(\Xc_b,\mathcal{O}_{\Xc_b})^G)$ is
the same for each $b \in A$.
\end{theorem}

\begin{proof}
For every $b \in A$, since $\Xc_b$ is not uniruled and each birational map in $G$ respects the fibers of $\pi$, we may apply \cite[Proposition 1.4(5) and Proposition 2.1]{CJS22} to show that the action of $G$ specializes to a birational action on $\Xc_b$. This action agrees with the restriction of the $G$-action on $\Xc$ to a neighborhood of the generic point of $\Xc_b$ (the action is defined there since $\Xc \rightarrow B$ is proper and the generic point of $\Xc_b$ has codimension $1$ in $\Xc$).

The morphism $\pi$ is flat and the dimension $h^{i}(\Xc_{b}, \O_{\Xc_{b}})$ is constant across $b \in B$ \cite[Theorem 2.62]{Kollar23}, so Grauert's theorem implies that $R^{i}\pi_{\ast}\O_{\Xc}$ is a vector bundle on $B$ and that the natural map
\[ (R^{i}\pi_{\ast}\O_{\Xc})_{b} \rightarrow H^{i}(\Xc_{b}, \O_{\Xc_{b}}) \]
is an isomorphism for all $b \in B$. We claim that there is a $G$-action on the vector bundle $R^{i}\pi_{\ast}\O_{\Xc}$ which commutes with the trivial action on the base $B$. To see this, note that since every fiber has rational singularities and $B$ is smooth, $\Xc$ has rational singularities \cite{Elkik78}. Let $\mu \colon \tilde{\Xc} \rightarrow \Xc$ be an equivariant resolution of singularities which regularizes the action of $G$. By the definition of rational singularities,
\[ R^{i}(\pi \circ \mu)_*\O_{\tilde{\Xc}} \cong R^{i}\pi_{\ast}\O_{\Xc}. \]
The left-hand side naturally has the structure of a $G$-equivariant vector bundle 
over $B$ (where $B$ carries the trivial $G$-action), so the right does as well.

We claim that for every $b \in A$, the specialization of the $G$-action to the fiber $\Xc_{b}$ is compatible with the $G$-action on $H^{i}(\Xc_{b}, \O_{\Xc_{b}})$. But this is true by the first paragraph of the proof. Since $R^i \pi_* \mathcal{O}_{\Xc}$ is a $G$-equivariant vector bundle, the fibers 
$(R^i \pi_* \mathcal{O}_{\Xc})_b$ are a family of $G$-representations over $B$.
The group $G$ is reductive, so we can apply \cite[Proposition 1]{Kraft89} 
to conclude that the representations in this family are constant up to isomorphism.
It follows that the dimension of the space of invariants $H^i(\Xc_b,\mathcal{O}_{\Xc_b})^G$, $b \in A$ is constant as well.
\end{proof}

It is important for our applications to allow for families with both ruled and 
non-ruled fibers in \Cref{ginvariants}, since this can occur in klt surface families.
For instance, it is possible for a family of klt surfaces to have rational special fiber with quotient singularities
and general fiber smooth of general type.  These types of smoothings have been
used to create general type surfaces with unusual properties 
(see \cite{LP07,PPS09a,PPS09b}).

\begin{remark}
\label{rem:forms}
When $\pi: \Xc \rightarrow B$ is a \textit{smooth} family satisfying the
assumptions of \Cref{ginvariants}, the global sections of $\Omega_{\Xc_b}^i$
and $\omega_{\Xc_b}^{\otimes i}$ also have constant dimension in $b$.  Indeed, 
Hodge numbers and plurigenera are deformation invariants of smooth complex
varieties. These global sections also inherit $G$-actions from the rational 
$G$-action on the family. A variation of our argument above using the pushforward 
of the relative contangent sheaf and its powers shows that 
the dimension of the $G$-invariant part of $H^0(\Xc_b,\Omega_{\Xc_b}^i)$ is 
constant in $b$ (and similarly for the pluricanonical bundles).  The 
$\Omega_{\Xc_b}^i$ case can also be deduced from \Cref{ginvariants} and 
Hodge symmetry.
\end{remark}

Finally, we will need a result describing the singularities of the total space of a family of klt varieties.

\begin{lemma}
\label{uniruledmod}
Suppse $\Xc \rightarrow B$ is a flat proper morphism such that every fiber is a klt variety.  Then the total
space $\Xc$ has uniruled modifications. 
\end{lemma}

\begin{proof}
Though the total space $\Xc$ may not be $\Q$-Gorenstein, we may apply 
\cite[Theorem 5.41]{Kollar23} to find a small birational modification 
$\mathcal{Y} \rightarrow \Xc$ such that $\mathcal{Y}$ is log canonical, and
in particular $\Q$-Gorenstein, in a neighborhood over the fiber $\Xc_0$ 
(see also \cite[Theorem 3.5(a)]{KSB88} in the case $\dim(\Xc/B) = 2$). 
Since the fiber $\mathcal{Y}_0$ is a normal birational modification of $\Xc_0$, it is klt, 
so inversion of adjunction gives that the total space
$\mathcal{Y}$ is also klt near $\mathcal{Y}_0$ \cite[Theorem 5.50]{KM98}.

Now let $E$ be an exceptional divisor over $\Xc$.  We may find a resolution
$Z$ of $\Xc$ containing $E$ which also admits a birational morphism 
$\mu: Z \rightarrow \mathcal{Y}$. Since
$\mathcal{Y} \rightarrow \Xc$ is small, $E$ is also exceptional over 
$\mathcal{Y}$.  By the proof of Shokurov's rational connectedness 
conjecture by Hacon and M\textsuperscript{c}Kernan, $E \cap \mu^{-1}(p)$ is rationally chain connected for any point $p \in \mathcal{Y}$ 
\cite[Corollary 1.6]{HM07}.
This implies that any point in $E \cap \mu^{-1}(p)$ is contained in a rational curve, so $E$ is uniruled \cite[pg. 180-181]{Kollar96}.
\end{proof}

\section{Proofs of main results}

Now we are ready to
prove \Cref{thm:main}.

\begin{proof}[Proof of \Cref{thm:main}]

For surfaces, note that uniruled is equivalent to ruled. Let $\pi \colon \Xc \rightarrow B$ 
be a family of integral klt surfaces with the
property that the very general fiber $\Xc_b$ admits a dominant degree
$d$ Galois cover of $\mathbb{P}^2$. For $b \in B$ such that the fiber $\Xc_{b}$ is ruled, 
we have $\irr(\Xc_{b}) \leq d$ by \Cref{ruled}. Hence, it suffices to prove \Cref{thm:main}
for fibers $b \in B$ where $\Xc_{b}$ is not uniruled.
Assume there exists such a fiber; then by Matsusaka's theorem
\cite[Theorem IV.1.6]{Kollar96},
the very general fiber is also not ruled.  We will denote
by $A \subset B$ the subset of closed points $b \in B$
where the fibers are not ruled (equivalently, not uniruled).
From now on, fix such a point $0 \in A$.

By \cite[Lemma 2.1]{Vial13}, a very general fiber $\Xc_b$
of $\pi$ is isomorphic, as an abstract scheme, to the
geometric generic fiber $\Xc_{\bar{\eta}}$ via a field
isomorphism $\bar{\eta} \cong \C$.  If there is a 
Galois rational map $\Xc_b \dashrightarrow 
\mathbb{P}^2_{\C}$ of degree $d$, we may use the isomorphism
$\bar{\eta} \cong \C$ to obtain isomorphisms $\Xc_{\bar{\eta}} \cong \Xc_b$ and 
$\mathbb{P}^2_{\C} \cong \mathbb{P}^2_{\bar{\eta}}$.
Composing these maps then gives a Galois
rational map of $\bar{\eta}$-varieties
$\Xc_{\bar{\eta}} \dashrightarrow 
\mathbb{P}^2_{\bar{\eta}}$. This rational map is defined on a finite extension
$\eta \subseteq L \subseteq \bar{\eta}$.  After replacing
$\Xc \rightarrow B$ with a suitable
base change (which doesn't affect the special fiber),
we may assume that the generic fiber of $\pi$ is a Galois
cover of $\mathbb{P}^2_{\eta}$.

Since the extension $K(\Xc_{\eta})/K(\mathbb{P}^2_{\eta})$
is Galois, its Galois group $G$ acts by field automorphisms
on $K(\Xc_{\eta})$ with fixed field $K(\mathbb{P}^2_{\eta})$.
It follows that each $g \in G$ gives an $\eta$-birational map 
$\Xc_{\eta} \dashrightarrow \Xc_{\eta}$, i.e., there 
is an inclusion $G \subseteq \mathrm{Bir}_{\eta}(\Xc_{\eta})$.
Spreading out the graph of each of these birational maps, 
we have an action $G \subseteq \Bir(\Xc/B)$. \Cref{uniruledmod} implies that $\Xc$ has \textit{sustained} ruled modifications.  Indeed, a base change does not alter any of the fibers, and uniruled implies ruled in dimension $2$.
Thus, we may apply \Cref{ginvariants} to see that the action of $G$ on $\Xc$ restricts to a birational action of $G$ on $\Xc_{0}$
(as well as every other non-ruled fiber),
and the family of $G$-representations 
$H^1(\Xc_b,\mathcal{O}_{\Xc_b}), b \in A$
is constant up to isomorphism.

We can regularize the $G$-action on $\Xc_{\eta}$
with a birational morphism
$\tilde{\Xc}_{\eta} \rightarrow \Xc_{\eta}$, so
that $G$ acts regularly on the smooth variety $\tilde{\Xc}_{\eta}$.
The quotient 
$\tilde{\Xc}_{\eta}/G$ is a rational variety, since its
field of fractions is $K(\Xc_{\eta})^G = K(\mathbb{P}^2_{\eta})$.

Again applying $G$-equivariant resolution of singularities, we can
spread out to find a regular model
$\tilde{\pi}: \tilde{\Xc} \rightarrow B$ with a regular $G$-action,
such that the generic fiber is 
isomorphic to $\tilde{\Xc}_{\eta}$, the very general fiber
is smooth, and the
special fiber $\tilde{\Xc}_0$ is an snc divisor in $\tilde{\Xc}$.
By spreading out $\varphi$, we also obtain an equivariant
birational map $\varphi_B: \tilde{\Xc} 
\dashrightarrow \Xc$ over $B$. This map is defined at every
codimension $1$ point of $\tilde{\Xc}$, so in
particular it is defined at the generic point of every
irreducible component of a fiber.  Since the rational
$G$-action on $\Xc$ respects fibers, it follows
that the $G$-action on $\tilde{\Xc}$ does as well.
Taking the quotient by $G$, we obtain a flat proper morphism
$\tilde{\Xc}/G \rightarrow B$, where 
$\tilde{\Xc}/G$ is normal and irreducible.  By
Matsusaka's theorem,
every component of the special fiber 
$\tilde{\Xc}_0/G$ is ruled.  One of these components is
birational to $U_0/G$, where $U_0 \subseteq \Xc_0$
is the open set where the rational $G$-action is defined.
Thus $U_0/G$ is ruled.

It remains to show that the variety $U_0/G$ is actually rational, 
so that $\mathrm{irr}(\Xc_0) \leq |G| = d$. For every $b \in A$,
choose a $G$-equivariant resolution of singularities $Y_b \rightarrow \Xc_b$
regularizing the $G$-action.

Since $\Xc_b$ has rational singularities, 
we know that $H^1(Y_b,\mathcal{O}_{Y_b}) \cong H^1(\Xc_b,\mathcal{O}_{\Xc_b})$ 
(and this isomorphism respects the $G$-action).
We claim that $H^1(Y_b,\mathcal{O}_{Y_b})^G = H^1(Y_b/G, \mathcal{O}_{Y_b/G})$.

Indeed, let $f: Y_b \rightarrow Y_b/G$ be the quotient morphism.  Then there
is a split injection $\mathcal{O}_{Y_b/G} \rightarrow f_* \mathcal{O}_{Y_b}$
identifying $\mathcal{O}_{Y_b/G}$ with the $G$-invariant subsheaf 
$(f_* \mathcal{O}_{Y_b})^G$.  
Finite morphisms are affine so the higher direct
images of $f$ vanish and we have $H^i(Y_b,\mathcal{O}_{Y_b}) = 
H^i(Y_b/G,f_*(\mathcal{O}_{Y_b}))$ for each $i > 0$ 
(see, e.g., \cite[Exercise III.8.2]{Hartshorne77}).  Hence
$$H^i(Y_b,\mathcal{O}_{Y_b})^G = H^i(Y_b/G,f_*\mathcal{O}_{Y_b})^G = 
H^i(Y_b/G,(f_* \mathcal{O}_{Y_b})^G) = H^i(Y_b/G,\mathcal{O}_{Y_b/G}).$$
Here the second equality holds because the $G$-invariants functor
is exact (since $G$ is finite and $|G|$ is invertible in $\C$) and 
$H^0(V,\mathcal{F}^G) = H^0(V,\mathcal{F})^G$ for any $G$-equivariant
quasicoherent sheaf $\mathcal{F}$ on $V$ over the trivial action on $V$.

The dimension of the group $H^i(Y_b,\mathcal{O}_{Y_b})^G$
is constant for all $b \in A$ by \Cref{ginvariants}, so the same is
true of the dimension of $H^i(Y_b/G,\mathcal{O}_{Y_b/G})$.
But for a very general $b$, $Y_b/G$ is rational so 
$H^i(Y_b/G,\mathcal{O}_{Y_b/G}) = 0$.  This vanishing therefore holds
for every $b \in A$, in particular when $b = 0$.
Thus $Y_0/G$, which is birational to $U_0/G$, is also rational, 
since it is ruled and has irregularity zero.
Therefore, the rational quotient map
$\Xc_0 \dashrightarrow \Xc_0/G$ by $G$ is a degree $d$ dominant rational map
to a rational surface, and $\irr(\Xc_0) \leq d$, as required.
\end{proof}

Since covers of degree $2$ are always Galois,
it follows that the property of having degree of
irrationality at most $2$ specializes in smooth families
of dimension at most $2$.

\begin{corollary}
\label{cor:spec_irr2}
Let $\pi: \Xc \rightarrow B$ be a flat proper
morphism such that every fiber is an integral klt surface. If the very general fiber has degree of irrationality at most $2$, then 
every fiber has degree of irrationality at most $2$.
\end{corollary}

We suspect that the above result is sharp, in the sense that the property of satisfying $\mathrm{irr}(X) \leq d$ might \textit{not} specialize in smooth families of surfaces in the non-Galois case. For instance, the example below describes a family of surfaces where the degree of irrationality is expected to jump from $3$ to $4$ under specialization.

\begin{example}\label{abel_surf}
Let $\pi: (\Ac, \Lc) \rightarrow B$ is a family of $(1,2)$-polarized abelian surfaces such that $B$ passes through some very general point of the moduli space of such abelian surfaces and some fiber $\Ac_0$ is a product $E_1 \times E_2$ of two non-isogenous elliptic curves. By Riemann-Roch, the linear series $|\Lc_{b}|$ is always a pencil. For the fibers $(\Ac_{b}, \Lc_{b})$ which are not products of elliptic curves, it is known that the general element of the pencil $|\Lc_{b}|$ is a smooth genus 3 curve \cite{Barth87}. For abelian surfaces containing a smooth genus 3 curve, Yoshihara has shown that their degree of irrationality is 3 \cite{Yoshihara96}. For most other polarization types, the degree of irrationality of the very general member is $4$ \cite{Chen19, Martin22}, and one might expect the same to be true of $\Ac_{0}$.
This question remains open. We note that 
\Cref{thm:main} does not apply in this setting because the degree $3$ map constructed by Yoshihara is not Galois.

\end{example}

\bibliographystyle{amsalpha}
\bibliography{Biblio}

\providecommand{\bysame}{\leavevmode\hbox to3em{\hrulefill}\thinspace}
\providecommand{\MR}{\relax\ifhmode\unskip\space\fi MR }
\providecommand{\MRhref}[2]{%
  \href{http://www.ams.org/mathscinet-getitem?mr=#1}{#2}
}
\providecommand{\href}[2]{#2}
\begin{thebibliography}{PPS09b}

\bibitem[Bar87]{Barth87}
Wolf Barth, \emph{Abelian surfaces with {$(1,2)$}-polarization}, Algebraic geometry, {S}endai, 1985, Adv. Stud. Pure Math., vol.~10, North-Holland, Amsterdam, 1987, pp.~41--84. \MR{946234}

\bibitem[Che19]{Chen19}
Nathan Chen, \emph{Degree of irrationality of very general abelian surfaces}, Algebra Number Theory \textbf{13} (2019), no.~9, 2191--2198. \MR{4039500}

\bibitem[CJS22]{CJS22}
Nathan Chen, Lena Ji, and David Stapleton, \emph{Fano hypersurfaces with no finite order birational automorphisms}, arXiv preprint arXiv:2208.07396 (2022).

\bibitem[CM23]{CM23}
Nathan Chen and Olivier Martin, \emph{Rational maps from products of curves to surfaces with {$p_g = q = 0$}}, Math. Z. \textbf{304} (2023), no.~4, Paper No. 62, 14. \MR{4617165}

\bibitem[CS20]{CS20}
Nathan Chen and David Stapleton, \emph{Fano hypersurfaces with arbitrarily large degrees of irrationality}, Forum Math. Sigma \textbf{8} (2020), Paper No. e24, 12. \MR{4096619}

\bibitem[CS24]{CS24}
\bysame, \emph{Rational endomorphisms of {F}ano hypersurfaces}, Selecta Math. (N.S.) \textbf{30} (2024), no.~1, Paper No. 16, 24. \MR{4690700}

\bibitem[Elk78]{Elkik78}
Ren\'ee Elkik, \emph{Singularit\'es rationnelles et d\'eformations}, Invent. Math. \textbf{47} (1978), no.~2, 139--147. \MR{501926}

\bibitem[Har77]{Hartshorne77}
Robin Hartshorne, \emph{Algebraic geometry}, Graduate Texts in Mathematics, vol. No. 52, Springer-Verlag, New York-Heidelberg, 1977. \MR{463157}

\bibitem[HM07]{HM07}
Christopher~D. Hacon and James M\textsuperscript{c}Kernan, \emph{On {S}hokurov's rational connectedness conjecture}, Duke Math. J. \textbf{138} (2007), no.~1, 119--136. \MR{2309156}

\bibitem[Kle05]{Kleiman05}
Steven~L. Kleiman, \emph{The {P}icard scheme}, Fundamental algebraic geometry, Math. Surveys Monogr., vol. 123, Amer. Math. Soc., Providence, RI, 2005, pp.~235--321. \MR{2223410}

\bibitem[KM98]{KM98}
J\'anos Koll\'ar and Shigefumi Mori, \emph{Birational geometry of algebraic varieties}, Cambridge Tracts in Mathematics, vol. 134, Cambridge University Press, Cambridge, 1998, With the collaboration of C. H. Clemens and A. Corti, Translated from the 1998 Japanese original. \MR{1658959}

\bibitem[Kol96]{Kollar96}
J\'anos Koll\'ar, \emph{Rational curves on algebraic varieties}, Ergebnisse der Mathematik und ihrer Grenzgebiete. 3. Folge. A Series of Modern Surveys in Mathematics [Results in Mathematics and Related Areas. 3rd Series. A Series of Modern Surveys in Mathematics], vol.~32, Springer-Verlag, Berlin, 1996. \MR{1440180}

\bibitem[Kol23]{Kollar23}
\bysame, \emph{Families of varieties of general type}, Cambridge Tracts in Mathematics, vol. 231, Cambridge University Press, Cambridge, 2023, With the collaboration of Klaus Altmann and S\'andor J. Kov\'acs. \MR{4566297}

\bibitem[Kra89]{Kraft89}
Hanspeter Kraft, \emph{{$G$}-vector bundles and the linearization problem}, Group actions and invariant theory ({M}ontreal, {PQ}, 1988), CMS Conf. Proc., vol.~10, Amer. Math. Soc., Providence, RI, 1989, pp.~111--123. \MR{1021283}

\bibitem[Kra18]{Kraft18}
\bysame, \emph{Regularization of rational group actions}, arXiv preprint arXiv:1808.08729 (2018).

\bibitem[KSB88]{KSB88}
J.~Koll\'ar and N.~I. Shepherd-Barron, \emph{Threefolds and deformations of surface singularities}, Invent. Math. \textbf{91} (1988), no.~2, 299--338. \MR{922803}

\bibitem[KT19]{KT19}
Maxim Kontsevich and Yuri Tschinkel, \emph{Specialization of birational types}, Invent. Math. \textbf{217} (2019), no.~2, 415--432. \MR{3987175}

\bibitem[LP07]{LP07}
Yongnam Lee and Jongil Park, \emph{A simply connected surface of general type with {$p_g=0$} and {$K^2=2$}}, Invent. Math. \textbf{170} (2007), no.~3, 483--505. \MR{2357500}

\bibitem[Mar22]{Martin22}
Olivier Martin, \emph{The degree of irrationality of most abelian surfaces is 4}, Ann. Sci. \'Ec. Norm. Sup\'er. (4) \textbf{55} (2022), no.~2, 569--574. \MR{4411875}

\bibitem[MH82]{MH82}
T.~T. Moh and W.~Heinzer, \emph{On the {L}\"uroth semigroup and {W}eierstrass canonical divisors}, J. Algebra \textbf{77} (1982), no.~1, 62--73. \MR{665164}

\bibitem[PPS09a]{PPS09a}
Heesang Park, Jongil Park, and Dongsoo Shin, \emph{A simply connected surface of general type with {$p_g=0$} and {$K^2=3$}}, Geom. Topol. \textbf{13} (2009), no.~2, 743--767. \MR{2469529}

\bibitem[PPS09b]{PPS09b}
\bysame, \emph{A simply connected surface of general type with {$p_g=0$} and {$K^2=4$}}, Geom. Topol. \textbf{13} (2009), no.~3, 1483--1494. \MR{2496050}

\bibitem[Tot16]{Totaro16}
Burt Totaro, \emph{Rationality does not specialise among terminal varieties}, Math. Proc. Cambridge Philos. Soc. \textbf{161} (2016), no.~1, 13--15. \MR{3505665}

\bibitem[Via13]{Vial13}
Charles Vial, \emph{Algebraic cycles and fibrations}, Doc. Math. \textbf{18} (2013), 1521--1553. \MR{3158241}

\bibitem[Yos96]{Yoshihara96}
Hisao Yoshihara, \emph{Degree of irrationality of a product of two elliptic curves}, Proc. Amer. Math. Soc. \textbf{124} (1996), no.~5, 1371--1375. \MR{1327053}

\end{thebibliography}

\end{document}